\documentclass[a4paper]{amsart}
\numberwithin{equation}{section}
\newtheorem{theorem}{Theorem}[section]
\newtheorem{corollary}{Corollary}[section]

\newtheorem{lemma}{Lemma}[section]
\newtheorem{example}{Example}[section]

\theoremstyle{remark}

\usepackage{amsmath, hyperref, amsfonts, amstext, amsgen, amsbsy, amsopn}
\usepackage{color}
\usepackage{graphicx}
\usepackage{subfigure}
\usepackage[margin=3cm]{geometry}
\title[Further results for a subclass of univalent functions]
 {Further results for a subclass of univalent functions related with differential equation}
\keywords{Univalent; Starlike; Convex; Close--to--convex; Fekete--Szeg\"{o} problem; Coefficient estimates; Toeplitz determinant.\\
*Corresponding Author}
\begin{document}
%----------------------------------------------------------------------------
\begin{abstract}
Let $\Omega$ denote the class of functions $f$ analytic in the open unit disc $\Delta$, normalized by the condition $f(0)=f'(0)-1=0$ and satisfying the inequality
\begin{equation*}
  \left|zf'(z)-f(z)\right|<\frac{1}{2}\quad(z\in\Delta).
\end{equation*}
The class $\Omega$ was introduced recently by Peng and Zhong (Acta Math Sci {\bf37B(1)}:69--78, 2017). Also let $\mathcal{U}$ denote the class of functions $f$ analytic and normalized in $\Delta$ and satisfying the condition
\begin{equation*}
  \left|\left(\frac{z}{f(z)}\right)^2f'(z)-1\right|<1\quad(z\in\Delta).
\end{equation*}
In this article, we obtain some further results for the class $\Omega$ including, an extremal function and more examples of $\Omega$, inclusion relation between $\Omega$ and $\mathcal{U}$, the radius of starlikeness, convexity and close--to--convexity and sufficient condition for function $f$ to be in $\Omega$. Furthermore, along with the settlement of the coefficient problem and the Fekete--Szeg\"{o} problem for the elements of $\Omega$, the Toeplitz matrices for $\Omega$ are also discussed in this article.
\end{abstract}
%----------------------------------------------------------------------------
\author[H. Mahzoon, R. Kargar]
       {H. Mahzoon and R. Kargar$^*$}
%----------------------------------------------------------------------------

%----------------------------------------------------------------------------
%----------------------------------------------------------------------------
%----------------------------------------------------------------------------
%----------------------------------------------------------------------------
\address{Department of Mathematics, Islamic Azad University, Firoozkouh
Branch, Firoozkouh, Iran}
\email {mahzoon$_{-}$hesam@yahoo.com {\it (H. Mahzoon)}}
%------------------------------------------------------------
\address{Young Researchers and Elite Club,
Ardabil Branch, Islamic Azad University, Ardabil, Iran}
       \email{rkargar1983@gmail.com {\it (R. Kargar)}}

%----------------------------------------------------------------------------
%----------------------------------------------------------------------------
\maketitle
%=================================================================
\section{Introduction}
%=================================================================
%\noindent
Let $\mathcal{A}$ denote the family of functions $f$ of the form
\begin{equation}\label{f(z)=z+a2z2+}
  f(z)=z+\sum_{n=2}^{\infty} a_n z^n\quad(z\in \Delta),
\end{equation}
which are analytic and normalized by the condition $f(0)=f'(0)-1=0$ in the open unit disc $\Delta:=\{z\in\mathbb{C}:|z|<1\}$ and $\mathcal{S}$ denote its subclass of univalent functions. We say that the function $f\in\mathcal{A}$ is starlike in $\Delta$ if $f(\Delta)$ is a set that starlike with respect to the origin. In other words, the straight line joining any point in $f(\Delta)$ to the origin lies in $f(\Delta)$. This means that $tz_0\in f(\Delta)$ when $z_0\in f(\Delta)$ and $0\leq t\leq1$. We denote this set of functions by $\mathcal{S}^*$. The well--known analytic description of starlike functions in terms of functions with positive real part states that $f\in\mathcal{S}^*$ if, and only, if
\begin{equation*}
  {\rm Re}\left\{\frac{zf'(z)}{f(z)}\right\}>0\quad(z\in\Delta).
\end{equation*}
We say that a set $\Lambda$ is convex if the line segment joining any two points in $\Lambda$ lies in $\Lambda$. This means that $tz_0+(1-t)z_1\in\Lambda$ where $z_0,z_1\in\Lambda$ and $0\leq t\leq1$. A function $f\in\mathcal{A}$ is called convex if $f(\Delta)$ is a convex set. The set of convex functions in $\Delta$ is denoted by $\mathcal{K}$. Analytically, $f\in\mathcal{K}$ if, and only, if
\begin{equation*}
  {\rm Re}\left\{1+\frac{zf''(z)}{f'(z)}\right\}>0\quad(z\in\Delta).
\end{equation*}
The classes $\mathcal{S}^*$ and $\mathcal{K}$ were introduced by Robertson, see \cite{ROB}. We have $\mathcal{K}\subset \mathcal{S}^*\subset\mathcal{S}$, see \cite{Duren}. A function $f\in\mathcal{A}$ is said to be close--to--convex function, if there exists a convex function $g$ such that
\begin{equation*}
  {\rm Re}\left\{\frac{f'(z)}{g'(z)}\right\}>0\quad(z\in\Delta).
\end{equation*}
 By the Noshiro--Warschawski theorem \cite[Theorem 2.16]{Duren}, every close--to--convex function is univalent. Also, this theorem is one of the important criterion for univalence.
Let $\mathcal{U}$ denote the class of all functions $f\in\mathcal{A}$ satisfying the following inequality
\begin{equation*}
  \left|\left(\frac{z}{f(z)}\right)^2f'(z)-1\right|<1\quad(z\in\Delta).
\end{equation*}
It is well--known that $\mathcal{U}\subset \mathcal{S}$, see \cite{Aksentiev}. It's worth mentioning that the Koebe function $k(z)=z/(1-z)^2$ belongs to the class $\mathcal{U}$ although functions in $\mathcal{U}$ are not necessarily starlike in $\Delta$, see \cite{FP2007, MP2007}.
For more details about the class $\mathcal{U}$ one can refer to \cite{kessiberian, MP2001, ObPo(AML), MP2013, MP2012, ozaki}.

Lately, Peng and Zhong \cite{peng2017}, introduced and discussed a new subclass of analytic functions as follows
\begin{equation*}
  \Omega:=\left\{f\in \mathcal{A}:\left|zf'(z)-f(z)\right|<\frac{1}{2}, z\in \Delta\right\}.
\end{equation*}
The class $\Omega$ is a subclass of the starlike functions \cite[Theorem 3.1]{peng2017}. The main motive for defining the class $\Omega$ is the relationship between the class $\Omega$ and the class $\mathcal{U}$, see for more details \cite{peng2017}.
Also, they have investigated some properties for the class $\Omega$, such as
\begin{itemize}
  \item growth and distortion theorem;
  \item $\Omega$ is a subset of the starlike functions;
  \item the radius of convexity;
  \item if $f,g\in \Omega$, then $f*g\in \Omega$, where "*" is the well--known Hadamard product;
  \item $\Omega$ is a closed convex subset of $\mathcal{A}$;
  \item and properties support point and extreme point of $\Omega$.
\end{itemize}
Peng and Zhong estimated the coefficients of function $f$ of the form \eqref{f(z)=z+a2z2+} belonging to the class $\Omega$, but there was no mention of the proof and its accuracy. In this article we give sharp estimates for the coefficients of functions $f$ belonging to the class $\Omega$.

Very recently, also Obradovi\'{c} and Peng (see \cite{OPeng2017}) studied the class $\Omega$ and obtained two sharp sufficient conditions for the function $f$ to be in the class $\Omega$ as follows:
\begin{itemize}
  \item if $|f''(z)|\leq1$, then $f\in \Omega$;
  \item if $|z^2f''(z)+zf'(z)-f(z)|\leq 3/2$, then $f\in \Omega$.
\end{itemize}
Following, we give another sufficient condition for functions $f$ to be in the class $\Omega$.

The function $f\in\mathcal{A}$ is subordinate to the function $g\in\mathcal{A}$, written as $f(z)\prec g(z)$ or $f\prec g$, if there exists an analytic function $w$, known as a Schwarz function, with $w(0)=0$ and $|w(z)|\leq|z|$, such that $f(z)=g(w(z))$ for all
$z\in\Delta$. Moreover, if $g\in \mathcal{S}$, then we have the following equivalence (c.f. \cite{MM-book})
\begin{equation*}
    f (z)\prec g(z) \Leftrightarrow f(0)=g(0)\quad {\rm and}\quad f (\Delta)\subset g(\Delta).
\end{equation*}

The structure of the paper is the following. In Section \ref{sec. exas lems} we give an extremal function for the class $\Omega$ and solve an open question related to the
inclusion relation between $\Omega$ and $\mathcal{U}$, partially. In Section \ref{sec. radius} some radius problems for the function $f\in\Omega$ are obtained. In Section \ref{sec. condition} we present two conditions for functions $f$ to be in the class $\Omega$. In Section \ref{sec. coeffi} we study the coefficients of the function $f$ of the form \eqref{f(z)=z+a2z2+} belonging to the class $\Omega$. Finally, in Section \ref{sec. fek-sze}, the Fekete-Szeg\"{o} problem and Toeplitz matrices are investigated.

\section{Extremal function and inclusion relation}\label{sec. exas lems}
%\noindent
First, we give an example for the class $\Omega$ which is an extremal function for several problems.
%**************************************
\begin{example}\label{exa. widetilde f}
  Let
 \begin{equation}\label{fwid}
  \widetilde{f}_n(z):=z+\frac{1}{2(n-1)}z^n \quad(n=2,3,\ldots, z\in\Delta).
 \end{equation}
  It is clear that $\widetilde{f}_n\in \mathcal{A}$ and
  \begin{equation*}
    z\widetilde{f}_n'(z)-\widetilde{f}_n(z)=\frac{1}{2}z^n\quad(n=2,3,\ldots, z\in\Delta).
  \end{equation*}
  Since $z\in\Delta$, thus $ |z\widetilde{f}_n'(z)-\widetilde{f}_n(z)|=|z^n/2|<1/2$ and consequently $\widetilde{f}_n\in \Omega$. We remark that the function $\widetilde{f}_n$ is univalent in $\Delta$ for $n=2,3,\ldots$. The function $\widetilde{f}_n$ is an extremal function for several problems such as, coefficient estimates, the radius of convexity and starlikeness in the class $\Omega$.
 \end{example}
 %********************************************
 This is an open question whether $\mathcal{U}\subset \Omega$ or $\Omega\subset\mathcal{U}$? We solve this question partially. The following Example \ref{ex2} shows that $\Omega\subset\mathcal{U}$ and $\mathcal{U}\not\subset\Omega$.
 %***************************************************
 \begin{example}\label{ex2}
    Let $\widetilde{f}_n$ $(n=2,3,\ldots)$ be defined by \eqref{fwid}. Then by Example \ref{exa. widetilde f} we have $\widetilde{f}_n\in\Omega$ for $n=2,3,\ldots$. In particular, if we take $n=2$, then the bounded analytic function $\widetilde{f}_2=z+z^2/2$ belongs to $\mathcal{U}$, see \cite[p. 175]{MP2001}. Also, the function $\widetilde{f}_3=z+z^3/4$ belongs to $\mathcal{U}$, too. Because
   \begin{equation*}
     \left(\frac{z}{\widetilde{f}_3(z)}\right)^2\widetilde{f}_3'(z)
     =\frac{1+3z^2/4}{(1+z^2/4)^2}
   \end{equation*}
   and for all $z\in\Delta$ we have
   \begin{equation}\label{exa. estimate}
     0\leq \left|\left(\frac{z}{\widetilde{f}_3(z)}\right)^2\widetilde{f}_3'(z)-1\right|
     =\left|\frac{1+3z^2/4}{(1+z^2/4)^2}-1\right|
     <0.56<1.
   \end{equation}
   %We remark that in order to find the number $0.56$ in the relation \eqref{exa. estimate} we used an applet called ComplexTool v 5.4 that can be found online at \url{http://www.jimrolf.com/complex.htm}.
   Now we consider the function $f_1$ % from \cite[Example 1]{ObPo(AML)}
   as follows
   \begin{equation*}\label{f lambda}
     \frac{z}{f_1(z)}=1+\frac{1}{2}z+\frac{1}{2}z^3\quad(z\in\Delta).
   \end{equation*}
   It is easy to see that $z/f_1(z)\neq0$ in $\Delta$ and
   \begin{equation*}
     \left(\frac{z}{f_1(z)}\right)^2f'_1(z)-1=-z^3.
   \end{equation*}
   Since $z\in\Delta$, thus $f_1$ belongs to $\mathcal{U}$. A simple calculation gives us
   \begin{equation*}
     zf'_1(z)-f_1(z)=-\frac{z^2+3z^4}{2+z+z^3}\quad(z\in\Delta).
   \end{equation*}
   Now, if we take $z_0=-2/3\in\Delta$, then
   \begin{equation*}
   |z_0f'_1(z_0)-f_1(z_0)|=\left|-\frac{z_0^2+3z_0^4}{2+z_0+z_0^3}\right|=1>1/2.
   \end{equation*}
   This shows that $f_1\not\in\Omega$. Therefore $\mathcal{U}\not\subset\Omega$.
 \end{example}
% {\bf Open problem.} Find the largest $r$ such that $\mathcal{U}\subset\Omega$ for all $z\in\Delta$, $|z|=r<1$.
 %********************************************
 \section{Radius problems}\label{sec. radius}
 Peng and Zhong \cite[Theorem 3.4]{peng2017} showed that the radius of convexity for the class $\Omega$ is $1/2$. Here, by use of the function \eqref{fwid}, we show that the result of  Peng and Zhong is sharp.
 \begin{example}
   The function $\widetilde{f}_n$ shows that the members of the class $\Omega$ are convex in the open disc $|z|<r$ where $r<1/2$. Thus the result of Theorem 3.4 of \cite{peng2017} is sharp.
 \end{example}
 \begin{proof}
   Let $\widetilde{f}_n$ be given by \eqref{fwid}. With a simple calculation, we get
   \begin{equation*}
     1+\frac{z\widetilde{f}_n''(z)}{\widetilde{f}_n'(z)}
     =1+\frac{n}{2}\frac{z^{n-1}}{1+\frac{n}{2(n-1)}z^{n-1}}\quad(n=2,3,\ldots).
   \end{equation*}
   Using the analytic definition of convexity, the radius
$r$ of convexity is the largest number $0<r<1$ for which
\begin{equation*}
  \min_{|z|=r}{\rm Re}\left\{1+\frac{z\widetilde{f}_n''(z)}{\widetilde{f}_n'(z)}\right\}\geq0.
\end{equation*}
Now, for every $r\in(0,1)$, we have
\begin{align*}
  {\rm Re}\left\{1+\frac{z\widetilde{f}_n''(z)}{\widetilde{f}_n'(z)}\right\}&={\rm Re}\left\{1+\frac{n}{2}\frac{z^{n-1}}{1+\frac{n}{2(n-1)}z^{n-1}}\right\}\\
  &\geq
  1-\frac{n}{2}\left|\frac{z^{n-1}}{1+\frac{n}{2(n-1)}z^{n-1}}\right| \\
  &\geq 1-\frac{n}{2}\frac{|z|^{n-1}}{1-\frac{n}{2(n-1)}|z|^{n-1}}\\
  &=1-\frac{n}{2}\frac{r^{n-1}}{1-\frac{n}{2(n-1)}r^{n-1}}=:\phi(r)\quad(|z|=r<1).
\end{align*}
It is easy to see that $\phi(r)>0$ if and only if
\begin{equation*}
  r<\left(\frac{2(n-1)}{n^2}\right)^{\frac{1}{n-1}}=:r_0\quad(n=2,3,\ldots).
\end{equation*}
We note that if we put $n=2$, then $r_0$ becomes $1/2$ and if $n\rightarrow \infty$, then $r_0\rightarrow 1$. This is the end of proof.
 \end{proof}
 %********************************************
 In the next result, with other proof we show that $\Omega\subset\mathcal{S}^*$.
 %**********************************
\begin{lemma}\label{lem. rad. star}
  Every function $f\in\Omega$ is a starlike univalent function in the open unit disc $\Delta$.
\end{lemma}
%***************************
\begin{proof}
 By \cite[Eq. (3.4)]{peng2017}, $f$ belongs to the class $\Omega$ if, and only if,
  \begin{equation}\label{rep. f}
    f(z)=z+\frac{1}{2}z\int_{0}^{z}\varphi(\zeta){\rm d}\zeta,
  \end{equation}
  where $\varphi\in \mathcal{A}$ and $|\varphi(z)|\leq1$ ($z\in\Delta$).
  Now from \eqref{rep. f}, we have
  \begin{equation*}
    \frac{zf'(z)}{f(z)}=1+\frac{1}{2}\frac{z^2 \varphi(z)}{z+\frac{1}{2}z^2\int_{0}^{1}\varphi(zt){\rm d}t}\quad(0\leq t\leq1).
  \end{equation*}
  Therefore by the analytic definition of starlikeness, we get
  \begin{align*}
    \min_{|z|=r}{\rm Re}\left\{\frac{zf'(z)}{f(z)}\right\} &\geq 1-\frac{1}{2}
    \left|\frac{z^2 \varphi(z)}{z+\frac{1}{2}z^2\int_{0}^{1}\varphi(zt){\rm d}t}\right|\\
    &>1-\frac{r}{2-r}\quad(|z|=r<1).
  \end{align*}
  It is easy to see that $1-\frac{r}{2-r}>0$ when $0<r<1$ and concluding the proof.
\end{proof}
In the sequel, we will show that the class $\Omega$ is a subclass of close--to--convex functions.
 %*************************************
 \begin{lemma}
   Every function $f\in\mathcal{A}$ which belongs to the class $\Omega$ is close--to--convex in $\Delta$.
 \end{lemma}
 \begin{proof}
   Let the function $f\in\mathcal{A}$ belongs to the class $\Omega$. Then by \eqref{rep. f}, we get
   \begin{equation*}\label{rep. f'}
     f'(z)=1+\frac{1}{2}z\varphi(z)+\frac{1}{2}z\int_{0}^{1}\varphi(z t)dt\quad(0\leq t\leq1).
   \end{equation*}
   Since $|\varphi(z)|\leq1$, we have
   \begin{align*}
     {\rm Re}\{f'(z)\} &={\rm Re}\left\{1+\frac{1}{2}z\varphi(z)+\frac{1}{2}z\int_{0}^{1}\varphi(z t)dt\right\} \\
     &\geq 1-\left|\frac{1}{2}z\varphi(z)+\frac{1}{2}z\int_{0}^{1}\varphi(z t)dt\right|\\
     &\geq1-r\quad(|z|=r<1).
   \end{align*}
   The last inequality is non--negative if $r<1$ and concluding the proof.
 \end{proof}
 %**************************************
 \section{Conditions for functions $f$ to be in $\Omega$}\label{sec. condition}
First, we give a sufficient condition for functions $f$ of the form \eqref{f(z)=z+a2z2+} to be in the class $\Omega$. We remark that since $\Omega$ is a subclass of the close--to--convex univalent functions, the following lemma also is a sufficient condition for univalence.
\begin{lemma}\label{lem suffi}
  Let $f\in\mathcal{A}$. If
  \begin{equation}\label{ineq. f in omega}
    \left|\left(\frac{f(z)}{z}\right)'\right|<\frac{1}{2}\quad(z\in\Delta, z\neq0),
  \end{equation}
  then $f\in\Omega$. The number $1/2$ is the best possible.
\end{lemma}
%********************************************
\begin{proof}
  Let $f\in\mathcal{A}$ satisfies the inequality \eqref{ineq. f in omega}. Since $z\in\Delta$ and consequently $|z|^2<1$, thus by the inequality \eqref{ineq. f in omega}, we get
  \begin{equation*}
    \left|z^2\left(\frac{f(z)}{z}\right)'\right|<\frac{1}{2}\quad(z\in\Delta, z\neq0).
  \end{equation*}
  Now the assertion follows from the following identity
\begin{equation}\label{indentity}
  z^2\left(\frac{f(z)}{z}\right)'=zf'(z)-f(z)\quad(z\in\Delta, z\neq0)
\end{equation}
and concluding $f\in\Omega$. For the sharpness, consider the function $\widehat{f}_\lambda(z)=z+\lambda z^2$ where $|\lambda|<1/2$ and $z\in\Delta$. A simple calculation gives that
\begin{equation*}
  \left(\frac{\widehat{f}_\lambda(z)}{z}\right)'=\lambda\quad(|\lambda|<1/2, z\in\Delta, z\neq0).
\end{equation*}
Therefore $\widehat{f}_\lambda\in\Omega$. It is easy to see that if $\lambda>1/2$, then $\widehat{f}_\lambda\not\in\Omega$. Also, since $\widehat{f}'_\lambda(z)=1+2\lambda z$ vanish at $-\frac{1}{2\lambda}$, we conclude that $\widehat{f}_\lambda$ is not univalent in $\Delta$ when $\lambda>1/2$. This is the end of proof.
\end{proof}
%**********************************************
%**********************************************
%\begin{remark}
 % We note that, if we consider the function $\widehat{f}_{1/2}(z)=z+ z^2/2$, then $z\widehat{f}'_{1/2}(z)-\widehat{f}_{1/2}(z)=z^2/2$ and thus $\widehat{f}_{1/2}$ belongs to the class $\Omega$. But, in this case we can not use the identity \eqref{indentity}. Because
 % \begin{equation*}
 %   \left(\frac{\widehat{f}_{1/2}(z)}{z}\right)'=\frac{1}{2}\not<\frac{1}{2}.
 % \end{equation*}
%\end{remark}
%********************************************
As an application of the Lemma \ref{lem suffi} we give another example for the class $\Omega$.
%*******************************
\begin{example}
  Define $f_{\gamma,\beta}(z)=z+\gamma z^2+\beta z^3$, where $\gamma$ and $\beta$ are two complex numbers. If $|\gamma|+2|\beta|<1/2$, then $f_{\gamma,\beta}(z)\in\Omega$. In particular, the function $\ell(z)=z+z^2/5+z^3/8$ belongs to the $\Omega$. We note that the function $\ell$ is univalent in the unit disc $\Delta$. The Figure \ref{fig:subfig01}(a) shows the image of $\Delta$ under the function $\ell(z)$.
\end{example}
%\begin{figure}[htp]
% \centering
% \includegraphics[width=7cm]{fpeng.pdf}\\
%\caption{The image of $\Delta$ under the function $\ell(z)=z+z^2/5+z^3/8$}\label{Fig:1}
%\end{figure}
%**********************************
%********************************************
Applying Lemma \ref{lem suffi}, we present a sufficient condition for the function $f_n(z)=z+a_n z^n$ ($n=2,3,\ldots$) to be in the class $\Omega$.
\begin{example}
  Consider the function $f_n(z)=z+a_n z^n$ where $z\in\Delta$. If
  \begin{equation}\label{an lemma}
    |a_n|<\frac{1}{2(n-1)}\quad(n=2,3,\ldots),
  \end{equation}
  then $f_n\in\Omega$.
\end{example}
%*************************
\begin{proof}
  Let $f_n(z)=z+a_n z^n$ and the inequality \eqref{an lemma} holds. From \eqref{an lemma}, we get
  \begin{equation}\label{(n-1) an}
    (n-1)|a_n||z|^{n-2}<\frac{1}{2}\quad(z\in\Delta, n=2,3,\ldots).
  \end{equation}
  Since
  \begin{equation}\label{fn z prime}
    \left(\frac{f_n(z)}{z}\right)'=(n-1)a_n z^{n-2}\quad(z\in\Delta, n=2,3,\ldots),
  \end{equation}
  the inequalities \eqref{(n-1) an} and \eqref{fn z prime}, imply that
  \begin{equation*}
     \left|\left(\frac{f_n(z)}{z}\right)'\right|<\frac{1}{2}\quad(z\in\Delta, n=2,3,\ldots).
  \end{equation*}
  Now the desired result follows from the Lemma \ref{lem suffi}.
\end{proof}

In the sequel, we recall from \cite[p. 24]{MM-book}, the function $q_M(z)$ given by
\begin{equation}\label{q(z)}
  q_M(z)=M\frac{Mz+a}{M+\overline{a}z}\qquad(z\in \Delta),
\end{equation}
where $M>0$, $|a|<M$. We have $q_M(0)=a$ and $q_M(\Delta)=\{\zeta: |\zeta|<M\}=:\Delta_M$. By applying the function \eqref{q(z)} and by the subordination, we present a sufficient and necessary condition for functions $f$ to be in the class $\Omega$.
\begin{lemma}\label{lem sub f in Omega}
  Let $f\in\mathcal{A}$. Then $f\in \Omega$ if, and only if,
  \begin{equation}\label{sub f in Omega}
    zf'(z)-f(z)\prec q_{1/2}(z)\quad(z\in \Delta),
  \end{equation}
  where $q_{1/2}(z)$ is defined by \eqref{q(z)} when $M=1/2$ and $a=0$.
\end{lemma}
%+++++++++++++++++++++++++++++++++++++
\begin{proof}
  If $f\in \Omega$, then by definition we have
  \begin{equation}\label{ineq. lem sub}
    |zf'(z)-f(z)|<\frac{1}{2}\quad(z\in \Delta).
  \end{equation}
  Thus, by \eqref{ineq. lem sub}, $zf'(z)-f(z)$ lies in the open disc $\Delta_{1/2}$ and it is clear that $q_{1/2}(\Delta)=\Delta_{1/2}$. Because $q_{1/2}$ is univalent, thus by the subordination principle, we get \eqref{sub f in Omega}. Indeed, since $f(z)=z+\sum_{n=2}^{\infty}a_n z^n$ and $zf'(z)-f(z)=\sum_{n=2}^{\infty}(n-1)a_n z^n$, we have $a=0$ in \eqref{q(z)}. This is the end of proof.
\end{proof}

%*******************************************
\section{On coefficients}\label{sec. coeffi}
%*******************************************
%\noindent
The first result of this section is the following.
\begin{theorem}\label{th. let phi}
  Let $\phi(z)=1+\sum_{n=1}^{\infty}c_n z^n$ be an analytic function in $\Delta$ and that satisfy the coefficient condition
  \begin{equation}\label{suffi condi phi}
    \sum_{n=1}^{\infty}n|c_{n}|<\frac{1}{2}.
  \end{equation}
  Then the function $f(z)=z\phi(z)$ belongs to the class $\Omega$.
\end{theorem}
%----------------------------------------------------------
\begin{proof}
  Let $f$ be given by $f(z)=z\phi(z)$ where $\phi(z)=1+\sum_{n=1}^{\infty}c_n z^n$. We have
  \begin{equation}\label{frac f z=phi}
    \frac{f(z)}{z}=\phi(z)=1+\sum_{n=1}^{\infty}c_n z^n.
  \end{equation}
  Now by differentiating the above relation \eqref{frac f z=phi} and multiplying by $z^2$, we get
  \begin{equation*}
    z^2\left(\frac{f(z)}{z}\right)'=\sum_{n=1}^{\infty}n c_n z^{n+1}.
  \end{equation*}
  Therefore using the coefficient condition \eqref{suffi condi phi} and the identity \eqref{indentity}, we deduce that
  \begin{equation*}
    |zf'(z)-f(z)|=\left|z^2\left(\frac{f(z)}{z}\right)'\right|=\left|\sum_{n=1}^{\infty}n c_n z^{n+1}\right|<\sum_{n=1}^{\infty}n|c_n|<\frac{1}{2}
  \end{equation*}
  and concluding the proof.
\end{proof}
Theorem \ref{th. let phi} allows us to find many examples that belong to $\Omega$. For example, consider $\phi_1(z)=1-z/5-z^2/8$. We have $c_1=-1/5$, $c_2=-1/8$ and $c_3=c_4=\cdots=0$. Thus the coefficients of $\phi_1$ satisfy the condition \eqref{suffi condi phi} and we conclude that  the univalent function $f(z)=z\phi_1(z)=z-z^2/5-z^3/8$ belongs to the class $\Omega$. The Figure \ref{fig:subfig01}(b) shows the image of $\Delta$ under the function $\phi_1(z)$.
%*************************************************
\begin{figure}[!ht]
\centering
\subfigure[]{
\includegraphics[width=5cm]{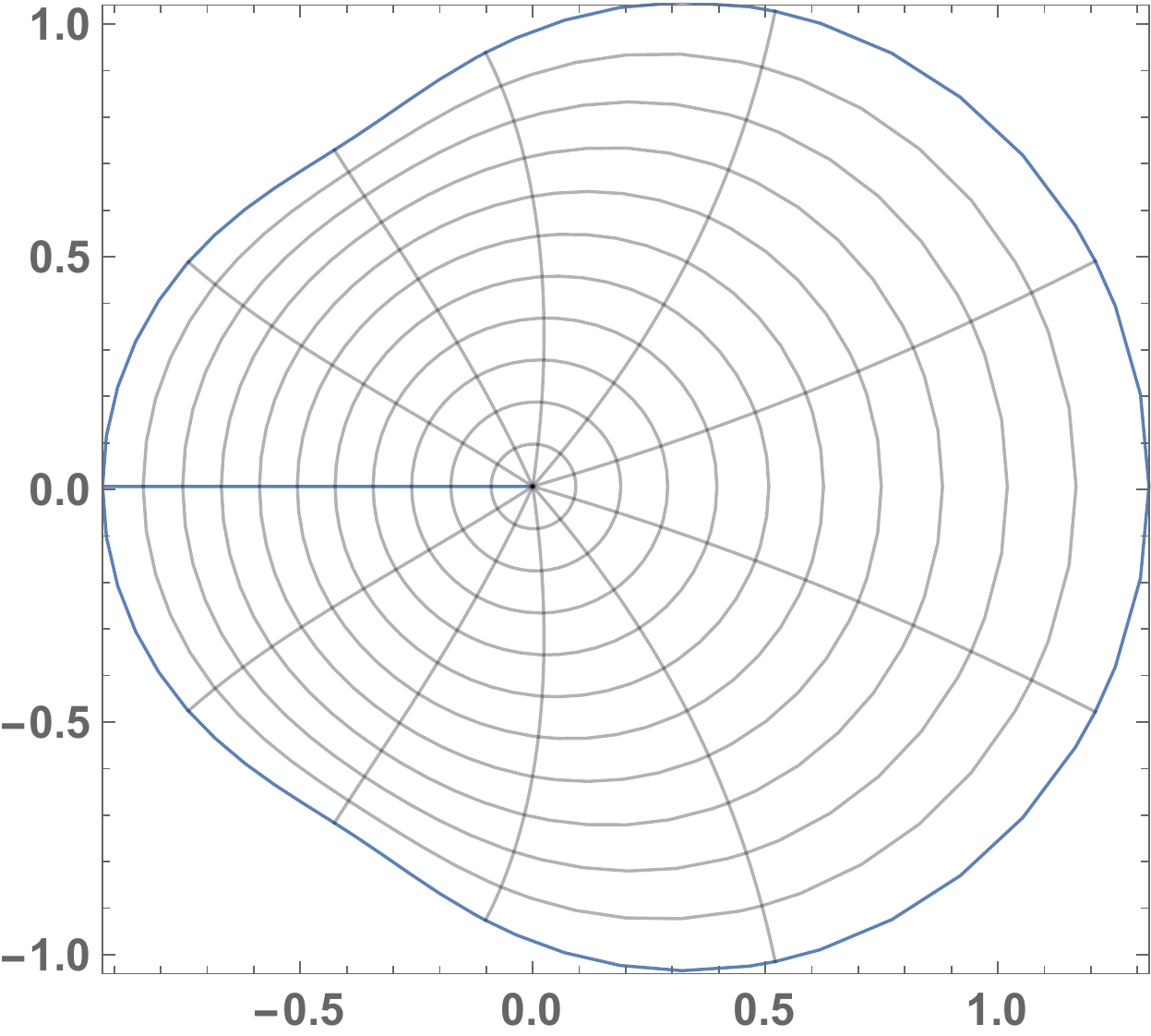}
	\label{fig:subfig1}
}
\hspace*{10mm}
\subfigure[]{
\includegraphics[width=5cm]{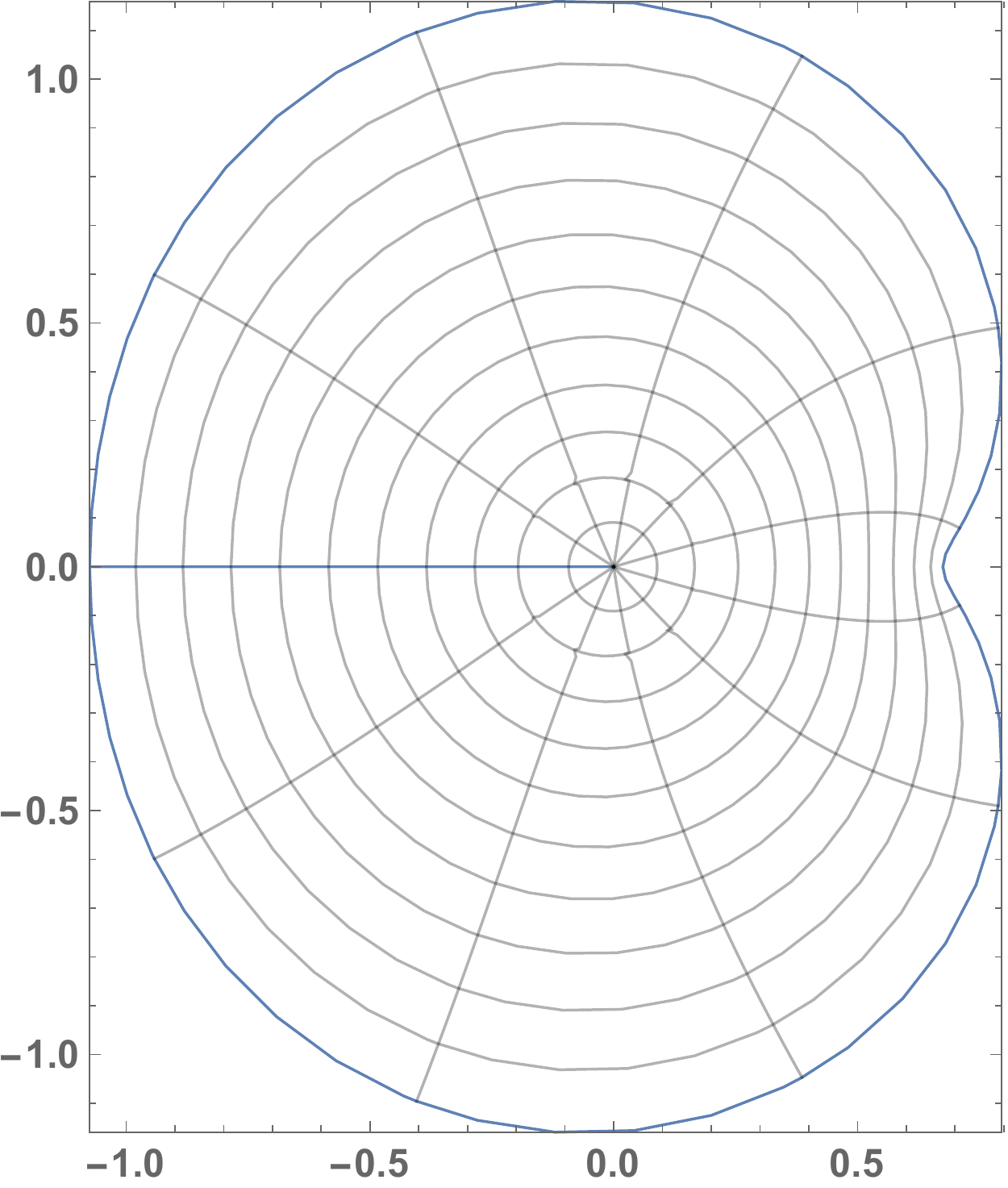}
	\label{fig:subfig2}
}

\caption[The boundary curve of $q_{1/2}(\Delta)$]
{\subref{fig:subfig1}: The image of $\Delta$ under the function $\ell(z)=z+z^2/5+z^3/8$  ،
 \subref{fig:subfig2}: The image of $\Delta$ under the function $\phi_1(z)=z-z^2/5-z^3/8$،
 }
\label{fig:subfig01}
\end{figure}
%*******************************************

 Peng and Zhong \cite[Corollary 3.12]{peng2017} said that (without proof and sharpness) the inequality \eqref{cef. inq.} (bellow) holds for the coefficients of functions $f$ belonging to the class $\Omega$. Here, by use of the Lemma \ref{lem sub f in Omega} we present a simple proof for \eqref{cef. inq.}. We remark that the inequality \eqref{cef. inq.} is sharp.

 The following lemma due to Rogosinski \cite[2.3 Theorem X]{Rog} helps to estimate of coefficients.
 %------------------------------------------------------------------------------------------
 \begin{lemma}\label{Rog Lemma}
   Let $q(z)=\sum_{n=1}^{\infty}Q_n z^n$ be analytic and univalent in $\Delta$ such that maps $\Delta$ onto a convex domain. If $p(z)=\sum_{n=1}^{\infty}P_n z^n$ is analytic in $\Delta$ and satisfies the subordination $p(z)\prec q(z)$, then $|P_n|\leq|Q_1|$ where $n=1,2,\ldots$.
 \end{lemma}

 \begin{theorem}\label{th. est coef}
   Let $f(z)=z+\sum_{n=2}^{\infty}a_n z^n\in \mathcal{A}$ be in the class $\Omega$. Then
   \begin{equation}\label{cef. inq.}
     |a_n|\leq \frac{1}{2(n-1)}\quad(n\geq2).
   \end{equation}
   The result is sharp.
 \end{theorem}
%****************************************
\begin{proof}
  Let $f$ be of the form \eqref{f(z)=z+a2z2+} belongs to the class $\Omega$. Then by Lemma \ref{lem sub f in Omega}, we have
  \begin{equation}\label{sub q1/2}
    zf'(z)-f(z)=\sum_{n=2}^{\infty}(n-1)a_n z^n\prec \frac{1}{2}z=:q_{1/2}(z)\quad(z\in\Delta).
  \end{equation}
  Since $q_{1/2}$ is convex univalent, thus applying the Lemma \ref{Rog Lemma} we get
  \begin{equation*}
    |(n-1)a_n|\leq \frac{1}{2}\quad(n\geq2).
  \end{equation*}
  Therefore the inequality \eqref{cef. inq.} holds. It is easy to see that the result is sharp for the function $\widetilde{f}_n$, where $\widetilde{f}_n$ is defined in \eqref{fwid}. This completes the proof.
\end{proof}
%*******************************************
\section{Fekete--Szeg\"{o} problem and Teoplitz matrices}\label{sec. fek-sze}
In recent years, the problem of finding sharp upper bounds for the Fekete--Szeg\"{o} coefficient functional associated with the $k$--th root transform has been studied by many scholars (see for example \cite{ALI}, \cite{KESBooth}, \cite{Sri}). For a univalent function $f$ of the form \eqref{f(z)=z+a2z2+}, the $k$--th root
transform is defined by
\begin{equation}\label{F(z)}
 F_k(z):=(f(z^k))^{1/k}=z+\sum_{n=1}^{\infty}b_{kn+1}z^{kn+1}\quad
 (z\in \Delta).
\end{equation}
A simple calculation gives that, for $f$ given by \eqref{f(z)=z+a2z2+},
\begin{equation}\label{FF(z)}
 (f(z^{k}))^{1/k}=z+\frac{1}{k}a_2z^{k+1}
 +\left(\frac{1}{k}a_3-\frac{1}{2}\frac{k-1}{k^2}a_2^2\right)z^{2k+1}+\cdots.
\end{equation}
Equating the coefficients of \eqref{F(z)} and \eqref{FF(z)}, we have
\begin{equation}\label{bk}
 b_{k+1}=\frac{1}{k}a_2\quad {\rm and}\quad b_{2k+1}=\frac{1}{k}a_3-\frac{1}{2}\frac{k-1}{k^2}a_2^2.
\end{equation}
 In the sequel, we obtain this problem for the class $\Omega$. Further we denote by $\mathcal{P}$ the well--known class of analytic functions $p$ with $p(0)=1$ and ${\rm Re}\{p(z)\}>0$ where $z\in\Delta$. Functions in $\mathcal{P}$ are called Carath\'{e}odory functions.
The following lemma due to Keogh and Merkes \cite{KM} will be useful in this section.
\begin{lemma}\label{FEK}
Let the function $p(z)$ given by
\begin{equation*}
 p(z)=1+p_1z+p_2z^2+\cdots,
\end{equation*}
be in the class $\mathcal{P}$. Then, for any complex number $\mu$
\begin{equation*}
 |p_2-\mu p_1^2|\leq 2\max\{1,|2\mu-1|\}.
\end{equation*}
The result is sharp.
\end{lemma}
%88***********************************************
\begin{theorem}\label{fekete-szego th.}
  Let the function $f$ of the form \eqref{f(z)=z+a2z2+} belongs to the class $\Omega$. Then for any complex number $\mu$ and $k\in\{1,2,3,\ldots\}$, we have
  \begin{equation}\label{fek k}
    \left|b_{2k+1}-\mu b_{k+1}^2\right|\leq \frac{1}{4k}\max \left\{1,\left|\frac{2\mu+k-1}{2k}\right|\right\},
  \end{equation}
  where $b_{2k+1}$ and $b_{k+1}$ are defined in \eqref{bk}.
  The result is sharp.
\end{theorem}
%****************************************
\begin{proof}
  If $f\in \Omega$, then by Lemma \ref{lem sub f in Omega} and definition of subordination there exits a Schwarz function $w(z)$ such that
  \begin{equation}\label{eq1}
    zf'(z)-f(z)=\frac{1}{2}z w(z)\quad(z\in\Delta).
  \end{equation}
  If we define the function $p$ as follows
  \begin{equation}\label{p(z)}
    p(z)=\frac{1+w(z)}{1-w(z)}=1+p_1z+p_2z^2+\cdots\quad(z\in\Delta),
  \end{equation}
  thus $p(z)$ is a analytic function in $\Delta$ and $p(0)=1$. A simple calculation gives us
  \begin{equation}\label{w(z)}
    w(z)=\frac{1}{2}p_1z+\frac{1}{2}\left(p_2-\frac{1}{2}p_1^2\right)z^2
    +\frac{1}{2}\left(p_3-p_1p_2+\frac{1}{4}p_1^3\right)z^3+\cdots\quad(z\in\Delta).
  \end{equation}
  From \eqref{eq1}--\eqref{w(z)}, equating coefficients gives, after simplification,
\begin{equation}\label{a2}
    a_2 = \frac{1}{4}p_1,
\end{equation}
    \begin{equation}\label{a3}
    2a_3 = \frac{1}{4}\left(p_2-\frac{1}{2}p_1^2\right).
    \end{equation}
Replacing \eqref{a2} and \eqref{a3} into \eqref{bk}, we get
  \begin{equation*}\label{b2k}
    b_{k+1}=\frac{1}{4k}p_1\quad{\rm and}\quad b_{2k+1}=\frac{1}{8k}
    \left(p_2-\frac{1}{2}p_1^2\right)-\frac{1}{32}\frac{k-1}{k^2}p_1^2.
  \end{equation*}
Thus
\begin{equation*}
  b_{2k+1}-\mu b_{k+1}^2=\frac{1}{8k}\left(p_2-\frac{1}{2}\left(\frac{3k+2\mu-1}{2k}\right)p_1^2\right).
\end{equation*}
If we let $\mu'=\frac{1}{2}\left(\frac{3k+2\mu-1}{2k}\right)$, then as an application of the Lemma \ref{FEK}, we get the desired inequality \eqref{fek k}.
\end{proof}
Putting $k=1$ in the Theorem \ref{fekete-szego th.}, we have.
%**************************************
\begin{theorem}\label{th. fekSz}
  (Fekete--Szeg\"{o} problem) Let $f$ be of the form \eqref{f(z)=z+a2z2+}
  belongs to the class $\Omega$. Then we have the following sharp inequality
\begin{equation*}
     \left|a_3-\mu a_2^2\right|\leq \frac{1}{4}\max \{1,|\mu|\}\quad(\mu\in \mathbb{C}).
   \end{equation*}
\end{theorem}
%************************************************************************
Since every function $f$ belongs to the class $\Omega$ is univalent, and every univalent function has an inverse $f^{-1}$, which is defined by $f^{-1}(f(z))= z$ ($z\in\Delta$) and
\begin{equation*}
  f(f^{-1}(w))=w\qquad (|w|<r_0;\ \ r_0\geq 1/4),
\end{equation*}
where
\begin{equation}\label{f-1}
  f^{-1}(w)=w-a_2w^2+(2a_2^2-a_3)w^3-(5a_2^3-5a_2a_3+a_4)w^4+\cdots,
\end{equation}
thus it is natural to consider the following result.
%corollary******************************************
\begin{corollary}\label{c3.4}
Let $f\in \mathcal{A}$ be in the class $\Omega$. Also
let the function $f^{-1}(w)=w+\sum_{n=2}^{\infty}b_nw^n$ be inverse of $f$. Then we have the following sharp inequalities
\begin{equation*}\label{b2}
  |b_i|\leq \frac{1}{2}\quad(i=2,3)
\end{equation*}
and $|b_4|\leq 19/24$.
\end{corollary}
\begin{proof}
  Relation \eqref{f-1} gives us
  \begin{equation*}
    b_2=-a_2,\, b_3=2a_2^2-a_3\quad {\rm and}\quad b_4=-(5a_2^3-5a_2a_3+a_4).
  \end{equation*}
  Thus, by putting $n=2$ in \eqref{cef. inq.} we can get
  \begin{equation*}
    |b_2|=|a_2|\leq 1/2.
  \end{equation*}
  For estimate of $|b_3|$, it suffices in Theorem \ref{th. fekSz}, we put $\mu=2$. Finally, since
  $b_4=5a_2(a_3-a_2^2)-a_4$ by the Fekete--Szeg\"{o} problem (Theorem \ref{th. fekSz}) for $\mu=1$, we get
  \begin{align*}
    |b_4| &=|5a_2(a_3-a_2^2)-a_4| \\
    &\leq 5|a_2||a_3-a_2^2|+|a_4|\\
    &\leq \frac{5}{2}\times\frac{1}{4}+\frac{1}{6}=\frac{19}{24}
  \end{align*}
   and concluding the proof.
\end{proof}
%**************************************
Following, we recall the symmetric Toeplitz determinant
\begin{equation*}
T_q(n):=\left|
 \begin{array}{cccc}
  a_n &a_{n+1} &\ldots &a_{n+q-1} \\
  a_{n+1} &a_n  &\ldots  &a_{n+q-2}  \\
  \vdots & \vdots & \vdots & \vdots \\
  a_{n+q-1}& a_{n+q-2} & \ldots &a_n  \\
 \end{array}
\right|,
\end{equation*}
where $n,q\in\{1,2,3,\ldots\}$ and $a_1=1$. Toeplitz matrices are one of the most well--studied and understood classes of structured matrices. Also, they have many applications in all branches of pure and applied mathematics (see for more details Ye and Lim \cite[Section 2]{ye}). In the next result, we obtain sharp bounds for the coefficient body $\left|T_q(n)\right|$, $q=2,3$ and $n=1,2$ where the entries of $T_q(n)$ are the coefficients of functions $f$ of form \eqref{f(z)=z+a2z2+} that are in the class $\Omega$.
%*************************************
\begin{theorem}
  Let $f$ be of the form \eqref{f(z)=z+a2z2+} belongs to the class $\Omega$. Then we have
  \begin{enumerate}
    \item  $|T_2(n)|\leq \frac{1}{4(n-1)^2}+\frac{1}{4n^2}$\quad($n\geq2$).
    %\item  $|T_2(2)|\leq \frac{5}{16}$.
    \item  $|T_3(1)|\leq \frac{13}{8}$.
    \item  $|T_3(2)|\leq \frac{329}{549}$.
  \end{enumerate}
  All the inequalities are sharp.
\end{theorem}
%***************************************
\begin{proof}
  (1) If $f\in \Omega$, then by Theorem \ref{th. est coef} and by the definition of the symmetric Toeplitz determinant $T_q(n)$ we have
  \begin{equation*}
    |T_2(n)|=\left|a_n^2-a_{n+1}^2\right|\leq\left|a_n\right|^2+\left
    |a_{n+1}\right|^2\leq \frac{1}{4(n-1)^2}+\frac{1}{4n^2}.
  \end{equation*}
  %(2) If we take $n=2$ in the case (1), then we get
   % \begin{equation*}
    %|T_2(2)|=\left|a_2^2-a_{3}^2\right|\leq\left|a_2\right|^2+\left
   % |a_3\right|^2\leq \frac{1}{4}+\frac{1}{16}=\frac{5}{16}.
  %\end{equation*}
  (2) It is clear that $|T_3(1)|=\left|1-2a_2^2+2a_3a_2^2-a_3^2\right|$. Thus
  \begin{align*}
    |T_3(1)|&\leq 1+2|a_2|^2+|a_3||2a_2^2-a_3| \\
    &\leq 1+2(1/4)+(1/4)(1/2)=13/8.
  \end{align*}
  Note that the Fekete--Szeg\"{o} problem is used, (see Theorem \ref{th. fekSz} with $\mu=2$).\\
  (3) We have
  \begin{equation*}
    T_3(2)=\left|
       \begin{array}{ccc}
         a_2 & a_3 & a_4 \\
         a_3 & a_2 & a_3 \\
         a_4 & a_3 & a_2 \\
       \end{array}
     \right|.
  \end{equation*}
  Thus we get $T_3(2)=(a_2-a_4)(a_2^2-2a_3^2+a_2a_4)$. For any $f\in \Omega$, we have $|a_2-a_4|\leq |a_2|+|a_4|\leq 1/2+1/6=2/3$. Now, it is enough to obtain the maximum of $|a_2^2-2a_3^2+a_2a_4|$ when $f\in \Omega$. By \eqref{a2}, \eqref{a3} and since
\begin{equation*}\label{a4}
    3a_4=\frac{1}{4}\left(p_3-p_1p_2+\frac{1}{3}p_1^3\right),
\end{equation*}
we get
\begin{align*}
  \left|a_2^2-2a_3^2+a_2a_4\right| &= \left|\frac{1}{16}p_1^2-\frac{1}{8}\left(p_2-\frac{1}{2}p_1^2\right)^2+\frac{1}{48}p_1
  \left(p_3-p_1p_2+\frac{1}{3}p_1^3\right)\right| \\
  &\leq \frac{1}{16}|p_1|^2+\frac{1}{8}\left|p_2-\frac{1}{2}p_1^2\right|^2+\frac{1}{48}
  |p_1||p_3-p_1p_2|+\frac{1}{144}|p_1|^4\\
  &\leq \frac{4}{16}+\frac{4}{8}+\frac{(2)(2)}{48}+\frac{16}{144}=\frac{329}{366}.
\end{align*}
Because $|p_3-p_1p_2|\leq2$ and $\left|p_2-\frac{1}{2}p_1^2\right|^2\leq 4$ (by Lemma \ref{FEK}). Thus $|T_3(2)|\leq \frac{2}{3}\times\frac{329}{366}=\frac{329}{549}$. Here, the proof ends.
\end{proof}

\end{document}